\documentclass[12pt]{amsart}
\usepackage{amssymb,latexsym}
\usepackage{enumerate}
\usepackage{amsmath}
\usepackage{amssymb}
\usepackage{amsfonts}
\usepackage{amsthm}
\usepackage{enumerate}
\usepackage{graphicx}
\usepackage{xcolor}
\usepackage{xy}
\usepackage{upgreek}
\usepackage{url}
\usepackage{ulem}

\makeatletter
\@namedef{subjclassname@2010}{%
  \textup{2010} Mathematics Subject Classification}
\makeatother

\theoremstyle{definition}

\newtheorem{Theorem}{Theorem}
\newtheorem*{Theorem*}{Theorem 1}
\newtheorem*{MTheorem*}{Main Theorem}
\newtheorem{Prop}[Theorem]{Proposition}

\newtheorem{Lemma}[Theorem]{Lemma}
\newtheorem{Cor}[Theorem]{Corollary}
\newtheorem*{Cor*}{Corollary}

\theoremstyle{definition}

\newtheorem{Remark}[Theorem]{Remark}

\numberwithin{equation}{section}

\newcommand{\gapit}{\noalign{\smallskip}}

\DeclareMathOperator{\Div}{div}

\newcommand{\mb}{\mathbb}

\newcommand{\ra}{\rightarrow}
\renewcommand{\l}{\ell}

\newcommand{\Z}{\mb{Z}}
\newcommand{\E}{\mathcal{E}}

\newcommand{\F}{\mb{F}}
\newcommand{\Fp}{\mathbb{F}_p} 

\newcommand{\leg}[2]{\left(\frac{#1}{#2}\right)}

\newcommand{\Q}{\mb{Q}}
\newcommand{\C}{\mb{C}}

\newcommand{\<}{\langle}
\renewcommand{\>}{\rangle}
\renewcommand{\S}{\mathcal{S}}

\begin{document}


\baselineskip=17pt


\title[Class Numbers via 3-Isogenies]{Class Numbers via 3-Isogenies\\ and Elliptic Surfaces}

\author[C. McLeman]{Cam McLeman}
\address{University of Michigan - Flint\\ Mathematics Department\\
Flint, MI 48439.}
\email{mclemanc@umflint.edu}

\author[D. Moody]{Dustin Moody}
\address{National Institute of Standards and Technology (NIST) \\ 100 Bureau Drive\\
Gaithersburg, MD, 20899}
\email{dbmoody25@gmail.com}

\date{}

\begin{abstract}
We show that a character sum attached to a family of 3-isogenies
defined on the fibers of a certain elliptic surface over $\Fp$ relates
to the class number of the quadratic imaginary number field
$\Q(\sqrt{-p})$.  In this sense, this provides a higher-dimensional
analog of some recent class number formulas associated to 2-isogenies
of elliptic curves.

\end{abstract}

\subjclass[2010]{Primary 11T24; Secondary 14J27, 11R29}

\keywords{elliptic curve, elliptic surface, class number, character sum}

\maketitle

\section{Introduction}
From the days of Diophantus, elliptic curves have long attracted the
interest of mathematicians.  More recently, elliptic curves have found
applications in diverse areas such as the proof of Fermat's last
theorem, factoring large integers, and in cryptography.  An area worth
particularly mentioning is the study of various character sums on the
points of an elliptic curve, which have been used to sum primes, find
generators for elliptic curve groups, and determine the structure of
such groups by demonstrating the uniformity of the distribution of
certain points (see, for example, \cite{AS}, \cite{BFGS}, \cite{JM},
\cite{KS}, \cite{LS}, \cite{PV}, \cite{Shp}, or \cite{Wil}).  A new
direction in this area has been to examine integer-weighted character
sums over elliptic curves \cite{MR11}, \cite{MR12}.  In this vein we
recall two results, whose interplay motivates the main result of this
paper.

First, to certain 2-isogenies $\tau$ of elliptic curves defined over
$\F_p$, McLeman and Rasmussen attach an integer-valued character sum
$S_\tau$, which is shown to be divisible by $p$ (see \cite{MR11}).  The
subsequent analysis of the quotient ${S_\tau}/{p}$ turns out to be of
arithmetic significance, providing a new class number formula
strikingly similar to a classical result of Dirichlet's.  Namely, we
have
\begin{align}\label{MRsum}
-\tfrac{1}{p}\,S_{\tau}=h_p^*,
\end{align}
where 
\[
h_p^*=\begin{cases}
h(\Q(\sqrt{-p}))&\text{ if }p\equiv 3\pmod{4}\\
0&\text{ otherwise.}
\end{cases}
\]
Here $p$ is a prime and $h(\Q(\sqrt{-p}))$ denotes the class number of $\Q(\sqrt{-p})$.

A second family of results from \cite{MR12} computes the mod-$p$
value of a much larger class of analogous character sums attached to
isogenies of larger degree, and specifically, finds several new
classes of character sums which are also evenly divisible by $p$.  In
light of the above class number formula, it seems of interest to
determine the analogous quotients.  The current article addresses the
analysis of such quotients, focusing on one particular family of
3-isogeny sums satisfying precisely this divisibility condition.  We
show that when this family of character sums is viewed as a single character
sum over an elliptic \textit{surface}, these quotients also compute
class numbers of quadratic imaginary number fields.  In this sense,
this can be viewed as a higher-dimensional analog of \eqref{MRsum}.

\textbf{Statement of Results --}
We begin with some notation.  Let $p$ and $\l$ be primes, with $p\equiv 1\pmod{\l}$.  Let $\tau\colon
E\to E'$ be an $\l$-isogeny of elliptic curves defined over the finite
field $\F_p$.   Let
$\zeta=\zeta_\ell$ denote a fixed primitive complex $\l$-th root of
unity, and choose a point $Q\in E'(\F_p)-\tau(E(\F_p))$.  We will see in Corollary \ref{cor2} that the particular choice of $Q$ does not matter.  From the
isomorphism $E'(\F_p)/\tau(E(\F_p))\cong \Z/\l\Z$, for each $P\in
E'(\F_p)$ we have $P-kQ\in \tau(E(\F_p))$ for a unique $0\leq k\leq
\l-1$.  We define the \textit{character $\chi_\tau$ attached to $\tau$}
by:
\[
\chi_\tau(P) = \zeta^k,\qquad\text{where }P-kQ\in \tau(E(\F_p)).
\]
In particular, $\chi_\tau(P)=1$ if and only if $P$ is in the image of
$\tau$.  

We adopt the following notation for lifting from $\F_p$ to $\Z$: for
$a\in\F_p$, let $\{a\}$ denote the unique integer $0\leq\{a\}\leq p-1$
such that $\{a\}\text{ mod }p=a$.  For the remainder of this paper we
will use $E(\Fp)$ to denote the set of affine $\Fp$-rational points on
the curve.  That is, we do not include the point $\infty$ in $E(\Fp)$.
Finally, given a point $P \in E(\Fp)$, let $x(P)$ be the
$x$-coordinate of $P$.

\bigskip

Notation in hand, we turn to more precise formulations of the two
aforementioned motivating results.  The first is the definition of the
character sum $S_\tau$ appearing in \eqref{MRsum}.  Let $\tau\colon
E\ra E'$ be a 2-isogeny defined over $\F_p$, and let $\chi_\tau$ be
the character attached to $\tau$ described above.  To such a $\tau$,
we introduce the integer-valued character sum

\begin{align}\label{sumdef}
S_{\tau}:=\sum\limits_{P\in E'(\F_p)}\{x(P)\}\chi_\tau(P).
\end{align}

For a concrete example of a situation in which \eqref{MRsum} holds, one can
take (\cite{MR11}, Proposition 2) $E$ and $E'$ to both be the curve
$y^2=(x+2)(x^2-2)$ over the finite field $\F_p$ for any prime $p>3$ of
good and ordinary reduction, and $\tau$ a degree-2 endomorphism of $E$
arising from complex multiplication on $E$ (by $\Z[\sqrt{-2}]$).  If we let $p=131$, for example, then it can be checked that $S_\tau=-655$, and indeed $655/131=5$ is the class number of $\Q (\sqrt{-131}).$

The second motivating result concerns a second class of isogenies
$\tau$ which give character sums divisible by $p$, but for which we do
\textit{not} know the analogous value of $S_{\tau}/p$.  To wit, let
$p\equiv 1\pmod{3}$, and let $E_d/\F_p$ be the elliptic curve
given by $y^2=x^3+d$. If we set $d'=-27d\in\F_p$, the function

\begin{align*}
\tau_d(x,y) := \left(\frac{y^2+3d}{x^2},\frac{y(x^3-8d)}{x^3}\right)
\end{align*}

\noindent defines a 3-isogeny from $E_d$ to $E_{d'}$.  Let us further
suppose that $\left(\frac{d}{p}\right)=1$ to isolate the interesting
cases of the character sum\footnote{If $d$ is a non-square mod $p$, 
  then $E_d$ has a $\F_p$-rational subgroup of order 3 but no
  $\F_p$-rational 3-torsion point.  The isogeny $\tau$ corresponding
  to this subgroup is now surjective, and the character $\chi_\tau$
  degenerates to the trivial character.}.  Then we have the following divisibility result:

\begin{Theorem}\label{divisibility}
Let $p$ and $\tau_d\colon E_d\to E_{d'}$ be as above, with character $\chi_{\tau_d}$.  Then
\begin{align*}
S_{\tau_d}:=\sum_{P\in E_{d'}(\F_p)}\{x(P)\}\chi_{\tau_d}(P)
\end{align*}
is an integer divisible by $p$.
\end{Theorem}
This theorem can be established using the techniques in \cite{MR12}.
We include a proof in this work, as it is a stepping stone to prove
our main result.  Worth mentioning is the following corollary of
Theorem \ref{divisibility}.
\begin{Cor}
\label{cor2}
Since $S_{\tau_d}\in\Z$, we note that $S_{\tau_d}$ is independent of the
choice of the point $Q\in E_{d'}(\F_p)-\tau_d(E_{d}(\F_p))$ used to define $\chi_{\tau_d}$.
\end{Cor}

Unlike the sum in \eqref{MRsum}, however, there does not seem to be a
direct relationship between the individuals sums $S_{\tau_d}$ and any
relevant class number.  Instead, we will see that class numbers arise
as a \textit{sum} of quotients $S_{\tau_d}/p$ as $d$ runs over the set
of all quadratic residues mod $p$.  In fact, the process of summing
over all such $d$ permits a concise reformulation in terms of elliptic
surfaces, which we describe now.  We begin with the substitution
$d=z^2$, so that running over all $z\in\F_p^\times$ is equivalent to
running over all square $d$ twice.  After the substitution, we arrive
at the algebraic surface
\[
\widetilde{\mathcal{E}}/\F_p\colon\quad y^2=x^3+z^2.
\]
More specifically, $\widetilde{\mathcal{E}}$ is an elliptic
surface, as the projection $\pi:\E\ra \mathbb{A}^1_z$ from $\E$ on to
the affine $z$-line equips the surface with a natural elliptic
fibration, with a single singular fiber over $z=0$.  Let
$\mathcal{E}=\widetilde{\mathcal{E}}-E_0$ denote the complement of the
singular fiber.  The above isogenies $\tau_d=\tau_{z^2}$ can now be
interpreted as maps from one fiber of $\E$ to another, namely from the
fiber over $z$ to the fiber over $-27z$.  We patch these
fiberwise-defined isogenies together to give a global endomorphism
$\tau\colon\E\to\E$ which respects the fibration in the sense that
$\pi(P_1)=\pi(P_2)$ implies $\pi(\tau(P_1))=\pi(\tau(P_2))$ for points
$P_1,P_2\in \mathcal{E}(\F_p)$.  We simply set
\begin{align*}
\tau(x,y,z)=\left(\frac{y^2+3z^2}{x^2},\frac{y(x^3-8z^2)}{x^3},-27z\right).
\end{align*}

We also extend the character $\chi_{\tau_d}$, defined \textit{a priori}
on each fiber to a global function on $\E$: For $P=(x,y,z)\in\E$, we have
$(x,y)\in E_{z^2}$, and hence it is sensible to write
\begin{align*}
\chi_\tau(P)=\chi_{\tau_{z^2}}(x,y).
\end{align*}

\noindent Note that, as before, $\chi_\tau(P)=1$ if and only if $P$ is
in the image of $\tau$.  Finally, we construct the higher-dimensional
character sum.  Define
\begin{align*}
\S_{\tau}=\sum_{P\in \mathcal{E}}\{x(P)\}\chi_\tau(P).
\end{align*}
  
\noindent The principal result of this work is the following theorem.

\begin{MTheorem*}\label{quotient}
Let $p\equiv 1\pmod{3}$, and let $\E$, $\tau$, $\chi_\tau$,
and $\S_{\tau}$ be as above.  Then
\[
\tfrac{1}{p}\,\S_{\tau}=h_p^*-\tfrac{p-1}{2}.
\]
\end{MTheorem*}

Our technique for calculating the sum $\S_{\tau}$ defined on $\E$ is
essentially a division of labor between two approaches. We will
independently sum ``vertically'' (over fibers) and ``horizontally''
(over sections) over our surface.  Section 2 deals with the former,
analyzing the fiber-wise isogenies $\tau_d$ defined in the
Introduction, and we address the global calculation of $\S_{\tau}$
in Section 3.

\section{Fiberwise-sums and the Tate pairing}

We maintain the notation established in the introduction.  Namely, we
have $p\equiv 1\pmod{3}$, a value $d\in\F_p^*$ with $\leg{d}{p}=1$,
and the 3-isogeny of $\F_p$-curves $\tau_d:E_d\to E_{d'}$ with
$d'=-27d$.  Note that the conditions on $p$ and $d$ imply
\[
\leg{d}{p}=\leg{-3}{p}=\leg{-3d}{p}=1.
\]

\noindent We begin with an analysis of the contribution to $\S_{\tau}$ coming from a
given fiber.  Set
\[
S_{\tau_d}:=\sum_{P\in E_{d'}(\F_p)}\{x(P)\}\chi_{\tau_d}(P),
\]
so that by viewing $\S_{\tau}$ as a sum over fibers, we see
\[\S_{\tau}=2\sum\limits_{\substack{d=0\\\leg{d}{p}=1}}^{p-1}S_{\tau_d}.\]

As in \cite{MR11}, the first step in evaluating $S_{\tau_d}$ is to use
the Tate pairing to provide explicit formulas for the computation of
$\chi_{\tau_d}$.  Let us briefly recall the construction of the
(complex-valued) Tate pairing attached to $\tau_d$, and its connection
to the character $\chi_{\tau_d}$.  Let $\widehat{\tau_d}$ be the dual
isogeny to $\tau_d$ and consider the point $T=(0,3\sqrt{-3d})$, which
generates of the group $E_{d'}[\widehat{\tau_d}](\F_p)$.  One begins
by finding a pair of functions $f_T$ and $g_T$ such that
$\Div(f_T)=3[T]-3[\infty]$ and $f_T\circ \tau_d=g_T^3$.  The Tate
pairing
\begin{equation*}
\psi_{\tau_d} \colon \frac{ E_{d'}(\F_p) }{ \tau(E_d(\F_p)) } \times
E_{d'}[\widehat{\tau_d}](\F_p) \longrightarrow \mu_3(\C),
\end{equation*}
is the (bilinear and non-degenerate) pairing given by the cubic
residue symbol
\begin{equation*}
\psi_{\tau_d}([P],kT)=\left(\frac{\psi'_{\tau_d}([P],kT)}{p}\right)_3,
\end{equation*}
where $[P]$ denotes the image of $P$ in the quotient
$E_{d'}(\F_p)/\tau_d(E_d(\F_p))$, and after choosing an arbitrary point $R\in E_{d'}(\F_p)-\<T\>$, we set
\begin{equation}
\label{tate}
\psi'_{\tau_d}( [P], kT) := \left\{ \begin{array}{ccl}
f(P)^k & & [P] \notin \<[T]\>  \\
\gapit
\left( \frac{ f(P + R) }{f(R)} \right)^k & & [P] \in \<[T]\>.
\end{array} \right.
\end{equation}
We use $\mu_3(\C)$ to denote the set of cubic roots of unity in $\C$.

\begin{Remark}
\label{ConjugateRemark}
The definition of the Tate pairing in \eqref{tate} actually outputs a
value in $\Fp^* /(\Fp^*)^3$, so we compose this version of the Tate
pairing with an isomorphism $\Fp^* /(\Fp^*)^3\cong \mu_3(\C)$.  In the
proof of Theorem \ref{ChiFormula}, we will choose the isomorphism
which forces the Tate pairing to coincide with our character
$\chi_{\tau_d}$.  
\end{Remark}

\begin{Prop}
Letting $E_d$, $E_{d'}$, $\tau_d$, and $T$ as above, we can take
\begin{align*}
f_{T}=y-3\sqrt{-3d}\qquad\text{ and }\qquad g_T=\frac{y-\sqrt{-3d}}{x}
\end{align*}
in the definition of the Tate pairing.
\end{Prop}
\begin{proof}
We easily check that $f_{T}$ is of degree 3 and vanishes only at $T$,
so $\Div(f_{T})=3[T]-3[\infty]$. Now we need only to verify that as
functions on $E_{d'}$, $f_{T}\circ\tau_d$ is the cube of $g_{T}$.  For a
point $P=(x,y)\in E_1(\F_p)$ (i.e., satisfying $x^3=y^2-d$), we have
\begin{align*}
f\circ\tau(P)&=\frac{y(x^3-8d)}{x^3}-3\sqrt{-3d}\\
&=\frac{y(y^2-9d)-3\sqrt{-3d}(y^2-d)}{x^3}\\
&=\left(\frac{y-\sqrt{-3d}}{x}\right)^3,
\end{align*}
as desired.
\end{proof}

\begin{Theorem}\label{ChiFormula}
With $E_d$, $\tau_d$, and $T$ as above, we have the following explicit
formulas for the character $\chi_{\tau_d}$:
\begin{equation}
\chi_{\tau_d}(P) = \psi_{\tau_d}([P], T) = \left\{
\begin{array}{ccl}
\bigl(\frac{-4d}{p}\bigr)^k_3 & & \text{if }[P] = [kT], \, k=0,1,2,\\
\gapit
\bigl( \tfrac{y-3\sqrt{-3d}}{p} \bigr)_3 & & \text{otherwise.}
\end{array}
\right.
\end{equation}
Note in particular that $\chi_{\tau_d}(P)=1$ if and only if
$P\in\tau_d(E_1(\F_p))$.
\end{Theorem}

\begin{proof}
Let us abbreviate $\tau=\tau_d$ for the duration of the proof.  For
the statement that $\chi_\tau(\cdot)=\psi_{\tau}([\cdot],T)$, we first
show that a point $P\in E_{d'}(\F_p)$ is in the image of $\tau$ if and
only if $\psi_{\tau}([P],T)=1$. By bilinearity, $\psi_{\tau}([P], kT)
= \psi_{\tau}([P], T)^k$. As~$E_2[\hat{\tau}]$ is generated by $T$,
$P$ pairs trivially with $T$ if and only if it pairs trivially with
every element of $E_2[\hat{\tau}]$. By the left non-degeneracy of
$\psi_\tau$, this occurs if and only if $[P]$ is the trivial class of
$E_2(\F_p)/\tau(E_1(\F_p))$, i.e., $P$ is in the image of $\tau$.
This shows that $\chi_{\tau}(\cdot)=\psi_{\tau}([\cdot],P)$ or
$\chi_{\tau}(\cdot)={\psi}^{-1}_{\tau}([.],P)$.  As in Remark
\ref{ConjugateRemark}, we now choose the correct isomorphism to
achieve equality.

We proceed to the second equality in the statement of the theorem.
Since the bottom case is the definition of the Tate pairing for such
points (given the calculation of $f_T$ from the previous proposition),
we only need to address the top case.  For this it suffices to show that
$\chi_\tau(T)=1$ if and only if $\left(\frac{-4d}{p}\right)_3=1$.  Let
$\alpha$ be a square root of $-3d$ in $\F_p$, such that $T=(0,3\alpha)$.  
Suppose first $\left(\frac{-4d}{p}\right)_3=1$, so that we have some $\delta \in \Fp$ with $\delta^3=-4d$.  Then 
\begin{align*} 
\tau(\delta,\alpha)&=\left(\frac{\alpha^2+3d}{\delta^2},\frac{\alpha(\delta^3-8d)}{\delta^3}  \right)\\
&=\left(0,\frac{-12d\alpha}{-4d} \right) \\
&=T,
\end{align*}
which shows that $\chi_\tau(T)=1$.

For the converse, we assume there exists $x,y \in \Fp$, with $\tau(x,y)=(0,3\alpha)$.  This requires that
$\frac{y^2+3d}{x^2}=0$ and $\frac{y(x^3-8d)}{x^3}=3\alpha.$  From the first of these equations we see that $y^2=-3d$.  As the point $(x,y)$ is on the curve $y^2=x^3+d$, it follows that $x^3=-4d$, so $\left(\frac{-4d}{p}\right)_3=1$.
\end{proof}

\noindent With these explicit formulas for $\chi_{\tau_d}$ in hand, we now prove
Theorem \ref{divisibility}.
\begin{Theorem*} \label{div2}
Let $p,E_{d}$, $E_{d'}$, and $\tau_d$ be as above.  Then
\begin{align*}
S_{\tau_d}:=\sum_{P\in E_{d'}(\F_p)}\{x(P)\}\chi_{\tau_d}(P)
\end{align*}
is an integer divisible by $p$.
\end{Theorem*}
\begin{proof}
Since the values of $\{x(P)\}$ are 0 for $P\in\<T\>$, they contribute
trivially to the sum (no matter the value of $\chi_{\tau_d}(P)$).  This allows us to
avoid breaking the sum into cases based on the results of the Tate
pairing.  We have {\small
\begin{align*}
\sum_{P\in E_{d'}(\F_p)}\{x(P)\}\chi_\tau(P)&=\sum_{P\in E_{d'}(\F_p)}\{x(P)\}\left(\frac{y(P)-3\sqrt{-3d}}{p}\right)_3\\
&=\sum_{y=0}^{p-1}\sum_{\substack{x=0\\x^3\equiv y^2+27d}}^{p-1}x\left(\frac{y-3\sqrt{-3d}}{p}\right)_3\\
&=\sum_{y=0}^{p-1}\left(\frac{y-3\sqrt{-3d}}{p}\right)_3\sum_{\substack{x=0\\x^3\equiv y^2+27d}}^{p-1}x.\\
\end{align*}
} Now each inner summand here is the sum of the lifts of the mod-$p$
cube roots of $y^2+27d$.  This sum is necessarily zero mod $p$ since the
coefficient of $x^2$ in the polynomial $x^3-(y^2+27d)$ is trivial.
Thus the whole sum is divisible by $p$, as desired.
\end{proof}

\section{The Global Sum}

We recall the global setting from the introduction. The surface
$\mathcal{E}$ is the complement of the singular fiber over $z=0$ in
the elliptic surface defined over $\F_p$ by $y^2=x^3+z^2.$ As such,
$\mathcal{E}$ is the union of fibers over $z$ for non-zero $z\in
\F_p$.  For $P=(x,y,z)\in\E$, we have $(x,y)\in E_{z^2}$, and we glue
together the fiber-wise isogenies $\tau_d$ to a global function $\tau$
and global character $\chi_\tau$ by defining
\begin{align*}
\chi_\tau(P)=\chi_{\tau_{z^2}}(x,y).
\end{align*}

The global sum $\S_{\tau}$ from the main theorem now decomposes as (twice)
the sum of the fiberwise-sums addressed in the previous section:
\begin{align*}
\S_{\tau}:=\sum_{P\in \mathcal{E}}\{x(P)\}\chi_\tau(P)=2\sum_{\substack{d\in\F_p\\\left(\frac{d}{p}\right)=1}} S_{\tau_{d}}.
\end{align*}

\noindent As an immediate corollary of Theorem \ref{divisibility}, we have the
divisibility of the global sum.
\begin{Cor}
With notation as previously defined, $p\mid \S_{\tau}$.
\end{Cor}

\noindent We now let $\beta$ denote a fixed square root of $-27$ (mod $p$), and
introduce the characteristic function
\[
e(x,y,z):=\begin{cases}
1&\text{ if }y^2\equiv x^3-27z^2\pmod{p},\\
0&\text{ otherwise}
\end{cases}
\]
on points $(x,y,z)\in\F_p^3$. The explicit computation of the Tate
pairing, using the function $f_T=y-3\sqrt{-3}\,z=y-\beta z$ in the
  fiber $E_{z^2}$ provides the following formula for $\S_{\tau}$.

{\small
\begin{align}\label{dSum}\nonumber
\S_{\tau}&=\sum_{z=1}^{p-1}\sum_{y=0}^{p-1}\sum_{\substack{x=1\\x^3\equiv y^2+27z^2}}^{p-1}x\left(\frac{y-\beta z}{p}\right)_3\\
&=\sum_{z=1}^{p-1}\sum_{y=0}^{p-1}\sum_{x=1}^{p-1}x\left(\frac{y-\beta z}{p}\right)_3e(x,y,z).
\end{align}
}

We re-arrange the orders of summation and use symmetries of the
function $e(x,y,z)$ to simplify the sum.  First, we write the sum as

\[
\sum_{x=1}^{p-1}x\,\sum_{y=0}^{p-1}\sum_{z=1}^{p-1}\left(\frac{y-\beta  z}{p}\right)_3e(x,y,z),
\]
and for a fixed $x$ and $y$ we address the innermost sum 
\[
s(x,y):=\sum_{z=1}^{p-1}\left(\frac{y-\beta  z}{p}\right)_3e(x,y,z).
\]
Note for a fixed $x$ and $y$, there exists a $z$ (with $1\leq z\leq
p-1$) such that $e(x,y,z)=1$ if and only if $\frac{y^2-x^3}{-27}$ is a
square mod $p$.  In this case there are precisely two such $z$'s, which we will denote by $\pm z_0$.  We then have
\begin{align*}
s(x,y)&=\left(\frac{y-\beta z_0}{p}\right)_3+\left(\frac{y+\beta z_0}{p}\right)_3\\
&=
\begin{cases}
1+1=2&\text{ if }\left(\frac{y-\beta z_0}{p}\right)_3=1,\\
\zeta+\zeta^{-1}=-1&\text{ otherwise.}
\end{cases}
\end{align*}
Here we have used that $\left(\frac{y-\beta z_0}{p}\right)_3$ and $\left(\frac{y+\beta z_0}{p}\right)_3$ are multiplicative inverses by the calculation
\[
\left(\frac{y-\beta z_0}{p}\right)_3\left(\frac{y+\beta z_0}{p}\right)_3=\left(\frac{y^2+27z_0^2}{p}\right)_3=\left(\frac{x^3}{p}\right)_3=1.
\]
Note that $z_0=z_0(x,y)$ can be written (only slightly abusively) as
$\sqrt{\frac{y^2-x^3}{-27}}$, and so we have $\beta
z_0=\pm \sqrt{y^2-x^3}$.  To summarize, the innermost sum $s(x,y)$ evaluates as one
of three possible cases depending on $x$ and $y$:
\begin{align}
\label{innersum}
s(x,y)&=
\begin{cases}
0&\text{ if }\left(\frac{y^2-x^3}{p}\right)\neq 1,\\
2&\text{ if }\left(\frac{y^2-x^3}{p}\right)=1\text{ and }\left(\frac{y\pm\sqrt{y^2-x^3}}{p}\right)_3=1,\\
-1&\text{ if }\left(\frac{y^2-x^3}{p}\right)=1\text{ and }\left(\frac{y\pm\sqrt{y^2-x^3}}{p}\right)_3\neq 1.\\
\end{cases}
\end{align}
Recall from equation \eqref{dSum}
\begin{equation} 
\label{finalsum}
\S_{\tau}=\sum_{x=1}^{p-1} x \sum_{y=0}^{p-1} s(x,y).
\end{equation}

\noindent Before we evaluate the new inner-most sum, we need a technical result.
\begin{Lemma}\label{equidistributed}
Let $x \neq 0$ be a fixed element of $\Fp$.  Then
\[
\left\vert \left\{y \in \Fp : \left(\frac{y^2-x^3}{p} \right) =1 \right\} \right\vert=
\begin{cases}
(p-1)/2&\text{ if }\left(\frac{x}{p}\right)=-1,\\
(p-3)/2&\text{ if }\left(\frac{x}{p}\right)=1.\\
\end{cases}\]

\end{Lemma}
\begin{proof}
It is easy to see the number of $u,v \in \Fp$ with $uv=x^3$ is $p-1$:
for any non-zero $u$, let $v=x^3/u$.  For each such solution, let
$y=(u+v)/2$ and $c=(u-v)/2$, which is equivalent to $u=y+c$ and
$v=y-c$.  Thus, the number of solutions to $(y+c)(y-c)=y^2-c^2=x^3$ is
also $p-1$.  We may rewrite this equation as $y^2-x^3=c^2$.

Now suppose first $x$ is not a square in $\Fp$.  Then $x^3$ is not a
square, and so there are no values of $y$ such that $y^2-x^3=0$.  For
each non-zero $c$, the values $\pm c$ are distinct and give the same
value for $c^2$.  We see there are $(p-1)/2$ values of $y$ for which
$ \left(\frac{y^2-x^3}{p} \right) =1$.

If instead $x$ is a square in $\Fp$, then so too is $x^3$ and there
will be two values of $y$ for which $y^2-x^3=0$.  This leaves $p-3$
solutions to $y^2-x^3=c^2$, with $c \neq 0$.  Again, as $\pm c$ both
lead to the same value of $c^2$, we find there are thus $(p-3)/2$
values of $y$ for which $ \left(\frac{y^2-x^3}{p} \right) =1$.
\end{proof}

\noindent With Lemma \ref{equidistributed}, we can now easily establish the following lemma.

\begin{Lemma}\label{theonlylemma}
For a fixed $x \neq 0$, 
\[ \sum_{y=0}^{p-1} s(x,y) = -1 -\left(\frac{x}{p}\right).\]
\end{Lemma}

\begin{proof}
Suppose first that $x$ is not a square.  We can ignore the values of $y$ for which $\left(\frac{y^2-x^3}{p}
\right) \neq 1$, as then $s(x,y)=0$ and they do not contribute to the
overall sum.  So we assume $y^2-x^3$ is a
non-zero square in $\Fp$.  We now claim that as we run over these
values of $y$, the values $y \pm \sqrt{y^2-x^3}$ are distinct.
If this were not the case, then there would exist a $y_1$ and $y_2$ such that
\[y_1 \pm \sqrt{y_1^2-x^3} = y_2 \pm \sqrt{y_2^2-x^3}.\]
Squaring both sides of this equation and simplifying, we see
\[y_1(y_1 \pm \sqrt{y_1^2-x^3}) = y_2(y_2 \pm \sqrt{y_2^2-x^3}) .\] As
$x \neq 0$, then $y_1 \pm \sqrt{y_1^2-x^3} \neq 0$, and similarly for
$y_2$.  By assumption, $y_1 \pm \sqrt{y_1^2-x^3} = y_2 \pm
\sqrt{y_2^2-x^3}$, and as this is non-zero we can divide through by it
to see that $y_1=y_2$.

From Lemma \ref{equidistributed}, there are $(p-1)/2$ values of $y$ with $\left(\frac{y^2-x^3}{p}
\right) = 1$.  Substituting these values into $y \pm
\sqrt{y^2-x^3}$ will result in $p-1$ distinct non-zero values in
$\Fp$.  It follows that there are $(p-1)/6$ values of $y$ such that $\left(\frac{y
  \pm \sqrt{y^2-x^3}}{p} \right)_3=1$ (or $\zeta$ or $\zeta^2$).  So
in this case, using \eqref{innersum} we see that

\[\begin{aligned} 
\sum_{y=0}^{p-1}s(x,y)& = 2\left(\frac{p-1}{6} \right)-\left(\frac{p-1}{3}  \right) \\
&= 0 \\
&=-1-\left(\frac{x}{p} \right). 
\end{aligned}\]

We similarly examine the case when $x$ is a square in $\Fp$.  By Lemma \ref{equidistributed} again, there are $(p-3)/2$ values of $y$ with $\left(\frac{y^2-x^3}{p} \right) = 1$.  These are the only values for which $s(x,y) \neq 0$.  As we run over them, then $y \pm
\sqrt{y^2-x^3}$ will run over $p-3$ distinct non-zero values in $\Fp$.
The two values not obtained are when $y = \pm x\sqrt{x}$, as then
$y^2=x^3$ and so $\left(\frac{y^2-x^3}{p} \right) = 0.$ But note that
then $y \pm \sqrt{y^2-x^3}$ just equals $y$, and $y = (\pm
\sqrt{x})^3$.  In short, the values of $\left(\frac{y \pm
  \sqrt{y^2-x^3}}{p} \right)_3$ will be equidistributed amongst $1$,
$\zeta$, and $\zeta^2$, except for the two values in $\Fp^*$, both
of which have it equaling 1.  So then,
\[
\begin{aligned} \sum_{y=0}^{p-1}s(x,y)& = 2\left(\frac{p-1}{6}-1 \right)-\left(\frac{p-1}{3}  \right) \\
&= -2 \\
&=-1-\left(\frac{x}{p} \right). \end{aligned}
\]
This completes the proof. 
\end{proof}
\noindent Finally, we can complete the proof of the main result.
\begin{Theorem}
We have
\[
\frac{\S_{\tau}}{p} = h_p^*-\tfrac{p-1}{2}.\]
\end{Theorem}
\begin{proof}
By equation \eqref{finalsum} and Lemma \ref{theonlylemma}, we see that
\[\begin{aligned} \S_{\tau}&=\sum_{x=1}^{p-1} x \left( -1 - \left(\frac{x}{p}\right) \right), \\
&=-\sum_{x=1}^{p-1}x\left(\frac{x}{p}\right) -\sum_{x=1}^{p-1}x, \\
&= ph_p^*-\tfrac{p(p-1)}{2}  . \end{aligned}\]
We have used Dirichlet's result that
\[\sum_{x=1}^{p-1} x\left(\frac{x}{p}\right) = -p h_p^*.\]
This completes the proof.
\end{proof}

\section{Conclusion}
It would be interesting to know whether other families of elliptic
curves or surfaces can be found which yield class number formulas
similar to the results in this paper.  All families of 3-isogenies
that we investigated which provided class number formulas ended up
being isomorphic to the curves $y^2=x^3+d$.  It would also be
interesting to find analogous formulas for curves with isogenies of
degree greater than 3.  As mentioned in the Introduction, some
divisibility properties have been shown in \cite{MR12}, however not
much is known about the corresponding quotients.  Finally, we leave it
as future work to analyze related character sums, where we weight by
other integer valued funtions other than $x(P)$.  Preliminary work
using $y(P)$ has been shown to have relations with class numbers.

\subsection*{Acknowledgements}

The authors would like to thank Christopher Rasmussen for some helpful
suggestions, and acknowledge the contribution of SAGE \cite{SAGE}, which
facilitated the construction of examples which were helpful in
discovering the main theorems of this work.

\end{document}